\documentclass[12pt]{amsart}

\usepackage{pdfsync}
\usepackage{xcolor}

\usepackage{geometry}
\usepackage{lscape}
\usepackage{graphicx}
\usepackage{epstopdf}    
\numberwithin{equation}{section}

\theoremstyle{plain}  
\newtheorem{thm}[equation]{Theorem}
\newtheorem*{main1}{Theorem \ref{main1}}
\newtheorem{prop}[equation]{Proposition}

\newtheorem{lemma}[equation]{Lemma}

\theoremstyle{definition}  
\newtheorem{defn}[equation]{Definition}

\newtheorem{remark}[equation]{Remark}

\usepackage{amsmath, amscd, amssymb, fullpage, setspace, , verbatim, lscape,hyperref, xypic}
\usepackage{amsthm}
\usepackage[shortlabels]{enumitem}
\usepackage[noadjust]{cite}

\usepackage{color}
\usepackage{lipsum}
\newcommand{\ind}{\hspace{15 pt}}

\newcommand{\ZZ}{\mathbb{Z}}
\newcommand{\Z}{\mathbb{Z}}

\newcommand{\CC}{\mathbb{C}}
\newcommand{\ER}{\mathbb{E}}

\newcommand{\colim}{\text{colim }}
\newcommand{\cp}{\mathbb{C}P^\infty}

\newcommand{\Maps}{\text{Maps}}
\newcommand{\MU}{\mathbb {MU}}

\newcommand{\E}{\mathbb{E}}

\renewcommand{\sp}[2]{\underline{#1}_{#2}}

\DeclareMathOperator{\BOg}{BO}

\DeclareMathOperator{\BUg}{BU}

\DeclareMathOperator{\MO}{MO}
\DeclareMathOperator{\MUR}{MR}

\DeclareMathOperator{\MUg}{MU}

\begin{document}
\pagestyle{plain}

\title
{Multiplicative structure on Real Johnson-Wilson theory}
\author{Nitu Kitchloo}
\address{Department of Mathematics, Johns Hopkins University, Baltimore, USA}
\email{nitu@math.jhu.edu}
\author{Vitaly Lorman}
\address{Department of Mathematics, University of Rochester, Rochester, USA}
\email{vlorman@ur.rochester.edu}
\author{W. Stephen Wilson}
\address{Department of Mathematics, Johns Hopkins University, Baltimore, USA}
\email{wsw@math.jhu.edu}
\thanks{The first author is supported in part by the NSF through grant DMS
  1307875.}
\date{\today}


\maketitle

\section{Introduction}
At the prime 2, Johnson-Wilson theory $E(n)$ \cite{JW73} is a complex-oriented cohomology theory which has a $C_2$-equivariant refinement, $\E(n)$ as a genuine $C_2$-equivariant spectrum, where the action of $C_2$ stems from complex conjugation. This was first constructed in \cite{HK01}, and Real Johnson-Wilson theory $ER(n)$ is defined to be the $C_2$-fixed points of $\E(n)$. The underlying nonequivariant spectrum of $\E(n)$ is Johnson-Wilson theory $E(n)$, and it is a homotopy associative, commutative, and unital ring spectrum. The goal of this note is to investigate whether the same properties hold for $\E(n)$ and $ER(n)$. 
\medskip
\noindent

Interest in this problem comes from the fact that $ER(n)$ is quickly becoming a useful and computable cohomology theory. For $n=1$ and 2, it reproduces familiar cohomology theories, $ER(1)=KO_{(2)}$ and $ER(2)=TMF_0(3)$ (the latter after suitable completion, see \cite{HM16}). The $ER(n)$-cohomology of a large (and growing) collection of spaces has been computed: real projective spaces and their products \cite{KW08a, KW08b, Ban13} (for $n=2$), complex projective spaces \cite{Lor16, KLW16b}, $\text{BO}$ and some of its connective covers \cite{KW14, KLW16b}, and half of all Eilenberg MacLane spaces \cite{KLW16a, KLW16b}. Furthermore, these computations have applications. In \cite{KW08a, KW08b}, the first and third authors used computations in $ER(2)$-cohomology to prove new nonimmersion results for real projective spaces.
\medskip
\noindent

The existence of a multiplicative structure on Real Johnson-Wilson theory has been suggested in a comment in \cite{HK01} (Comment 5 following the proof of Theorem 2.28) which claims that $\E(n)^{\star}(\E(n)\wedge\E(n))$ may be calculated and from this it may be shown that $\E(n)$ is a (homotopy) associative, commutative, and unital ring spectrum. The results in this note were born in the attempt to verify the above claim. Unfortunately, we were unsuccessful in doing so. However, we show that $\E(n)$ represents an $\MU$-algebra which is homotopy unital, associative, and commutative \emph{up to phantom maps}. By a phantom map, we mean a map $f: X \longrightarrow Y$ which has trivial restriction to any finite CW complex mapping into $X$. In addition, we show that the $\E(n)$-cohomology of an equivariant topological space is canonically a commutative ring. 

\medskip
\begin{thm}\label{main1}
$\E(n)$ is a homotopy commutative, homotopy associative, unital $\MU$-algebra up to phantom maps. In other words, there exist unit and multiplication maps:
\[ 1: \MU \longrightarrow \E(n), \quad \quad \hat{\mu} : \E(n) \wedge \E(n) \longrightarrow \E(n), \]
such that all the obstructions to $\hat{\mu}$ being a homotopy associative and homotopy commutative $\MU$-algebra structure are phantom maps. Differently said, all the corresponding structure diagrams commute up to phantom maps. Furthermore, the forgetful map:
\[ \rho : \E(n)^0(\E(n) \wedge \E(n)) \longrightarrow E(n)^0(E(n) \wedge E(n)), \]
maps $\hat{\mu}$ to the canonical product $\mu$ on the non-equivariant Johnson-Wilson spectrum $E(n)$. 
\end{thm}

\medskip

Theorem \ref{main1} tells us that the Real Johnson-Wilson theory is valued in commutative rings when applied to \emph{finite} CW complexes. Our second result extends this to the category of all spaces.
\medskip

\begin{thm}\label{main2}
With any choice of multiplication $\hat{\mu}$ as above, the spectrum $\E(n)$ represents a multiplicative cohomology theory on the category of $C_2$-spaces valued in (bigraded) commutative rings. There are natural transformations of ring-valued cohomology theories $\MU^\star(-) \longrightarrow \E(n)^\star(-)$ and $\E(n)^\star(-) \longrightarrow E(n)^*(-)$.
\end{thm}

The results of this document justify the assumption of commutativity in the computations of the $ER(n)$-cohomology of topological spaces made by the authors in previous work. We conclude by revisiting a result of the first and third authors concerning the $MR(n)$-orientation of $\MO[2^{n+1}]$ to correct an error in the proof of \cite[Theorem 1.4]{KW14}.

In the course of proving Theorem \ref{main2} we show that the infinite loop space underlying $ER(n)$, $\underline{ER(n)}_0$, is a homotopy commutative, associative, and unital $H$-ring space (Lemma \ref{zerospace}). Applying the Bousfield-Kuhn functor shows that the $K(n)$-localization, $L_{K(n)}ER(n)$, is in fact a homotopy commutative, associative and unital ring spectrum (not just up to phantom maps). The authors are grateful to Tyler Lawson for pointing this out. 

Hahn and Shi \cite{HS17} have recently proved that Johnson-Wilson theory $E(n)$ admits an $A_\infty$ ring structure for which the $C_2$-action is given by an $A_\infty$ involution (but not necessarily an $A_\infty$ ring automorphism). The question of (homotopy) commutativity is not readily addressed by their techniques. They show that these issues resolve on $K(n)$-localization giving rise to an $E_\infty$ ring spectrum structure on $L_{K(n)}ER(n)$. 

The present paper shows a (homotopy) commutative and associative ring structure up to
phantom maps, but to the authors' present knowledge, ER(n) is not (yet) known to be a
homotopy commutative and associative ring spectrum.

\medskip
\noindent
The authors would like to thank Neil Strickland for helpful discussion related to this document and the referee for diligent and valuable comments.

\section{Background}

\medskip

In this section, we recall a few background definitions and theorems from \cite{KW07b} and \cite{KLW16a} that we use in subsequent sections.

\medskip
\noindent
A genuine $C_2$-equivariant spectrum $\E$ is a family of $C_2$-spaces $\underline{\E}_{a+b\alpha}$, indexed over elements $a+b\alpha \in RO(C_2)$ where $\alpha$ denotes the sign representation, together with a compatible system of equivariant homeomorphisms
$$\xymatrix{ \underline{\E}_{a-r+(b-s)\alpha} \ar@{->}[r]^-{\simeq} & \Omega^{r+s\alpha} \underline{\E}_{a+b\alpha}}$$
where the right hand side denotes the space of pointed maps (endowed with the conjugation action) from the one point compactification of the representation $r+s\alpha$. The reader may refer to \cite{HK01} for more details on $C_2$-spectra. We will denote by $ER$ the homotopy fixed
point spectrum of the $C_2$-action on $\E$ and by $E$ the underlying nonequivariant spectrum given by forgetting the $C_2$-action. An example of a $C_2$- spectrum of interest to us is $\MU$, whose underlying nonequivariant spectrum is complex cobordism, $\MUg$, studied first by Landweber \cite{Lan68}, Fujii \cite{Fuj}, Araki and Murayama \cite{AM78}, and more recently by Hu-Kriz \cite{HK01}. The action of $C_2$ is induced by the complex conjugation action on the pre-spectrum representing $\MU$ in the usual way (see, e.g. \cite{HK01}). 

\medskip
\noindent
The $p=2$ Johnson-Wilson theory $E(n)$
lifts to a $C_2$-spectrum, $\E(n)$, defined as an $\MU$-module by coning off certain equivariant lifts of the Araki generators $v_i$ for $i > n$, and then inverting the lift of $v_n$. We shall call these equivariant lifts by the same names, $v_i$. The Real Johnson-Wilson theories, $ER(n)$, are defined as the fixed points of $\E(n)$.
\medskip
\noindent

Working with cohomological grading, let $Y$ be a $C_2$-space and let $\ER^{\ast(1+\alpha)}(Y)$ denote the subgroup of diagonal elements in the equivariant $\ER$-cohomology of $Y$ i.e. 
$$\ER^{\ast(1+\alpha)}(Y) := \pi_0\Maps^{C_2}(Y,\underline{\ER}_{\ast(1+\alpha)}).$$
 Consider the group homomorphism given by forgetting the $C_2$-action: 
\[ \rho : \ER^{\ast(1+\alpha)}(Y) \longrightarrow E^{\ast}(Y). \]
Notice that the image of $\rho$ belongs to the graded sub-group of elements in even degree. We recall the following definitions from \cite{KW07b}.

\medskip
\begin{defn} A $C_2$-space $Y$ is said to have the weak projective property with respect to a $C_2$-spectrum $\ER$ if the map
\[ \rho : \ER^{\ast(1+\alpha)}(Y) \longrightarrow E^{2\ast}(Y), \]
is an isomorphism of graded abelian groups.
\end{defn}

Note that in the case that the $C_2$-spectrum $\E$ is an $\MU$-module spectrum, the map $\rho$ in the definition above is a map of $\MU^{*(1+\alpha)}\cong MU^{2*}$-modules by virtue of the fact that the $\MU$-module structure on $\E$ forgets to an $MU$-module structure on $E$.

We will make extensive use of spaces with the weak projective property. In order to recognize such spaces, we need some auxiliary definitions.

\medskip
\begin{defn} A pointed $C_2$-space $X$ is said to be projective if 
\begin{enumerate}
\item $H_*(X; \ZZ)$ is of finite type. 
\item $X$ is homeomorphic to $\bigvee_I ({\cp})^{\wedge k_I}$ 
for some weakly increasing sequence of integers $k_I$, 
with the $C_2$ action given by complex conjugation. 
\end{enumerate}
\end{defn}

\medskip
\noindent
By a $C_2$-equivariant $H$-space, we shall mean an $H$-space whose multiplication map is $C_2$-equivariant.

\medskip
\begin{defn} A $C_2$-equivariant $H$-space $Y$ is said to have the projective 
property if there exists a projective space $X$, along 
with a pointed $C_2$-equivariant map $f : X \longrightarrow Y$, 
such that $H_*(Y;\ZZ/2)$ is generated as an algebra by the image of $f$.
\end{defn} 

The following theorem, proved in \cite[Theorem 2.6]{KLW16a}, establishes that spaces with the projective property have the weak projective property with respect to certain $\MU$-module spectra (in particular, $\ER(n)$).

\medskip
\begin{thm} \label{thm:weakproj}
Let $Y$ be a $C_2$-equivariant $H$-space with the projective property. Let $\ER$ denote any complete $\MU$-module spectrum with underlying spectrum $E$, satisfying the property that the forgetful map: $\rho^\ast : \ER^{\ast(1+\alpha)} \longrightarrow E^\ast$, is an isomorphism. Then the space $Y$ has the weak-projective property with respect to $\ER$. In other words, the following map is an isomorphism of $\MU^{\ast(1+\alpha)}$-modules: 
\[ \rho : \ER^{\ast(1+\alpha)}(Y) \longrightarrow E^{\ast}(Y). \]
\end{thm}

\medskip
\begin{remark}\label{smashproj}
The smash product of a finite collection of spaces with the projective property is an example of a space that has the weak projective property with respect to any $\E$ as in Theorem \ref{thm:weakproj}, but not the projective property (since it is not an $H$-space). This follows from writing $Y_1 \wedge Y_2=(Y_1 \times Y_2)/(Y_1 \vee Y_2)$ and the Five Lemma.
\end{remark}

\medskip
\noindent
There are many examples of spaces with the projective property. The ones of interest in this document will be $\underline{\ER(n)}_0$ and its products. 

\medskip
\begin{lemma}\label{zeroproj} $\underline{\E(n)}_0^{\times j} \times \underline{\MU}_{i(2^n-1)(1+\alpha)}^{\times s}$ has the weak projective property with respect to $\E(n)$ for all $i, j, s \geq 1$.
\end{lemma}
\begin{proof} By \cite[Theorem 1-4]{KW13}, the spaces $\sp{\E(n)}{0}$ and $\sp{\MU}{(i(2^n-1)(1+\alpha)}$ have the projective property.  Thus $\sp{\E(n)}{0}^{\times j} \times \underline{\MU}_{i(2^n-1)(1+\alpha)}^{\times s}$ is a restricted product of a family of spaces with the projective property (i.e. it is the colimit of finite products of spaces with projective property). A finite product of spaces with projective property evidently has the weak projective property, and since $\rho: \E(n)^{*(1+\alpha)}(-) \longrightarrow E(n)^{2*}(-)$ is an isomorphism at each stage, it follows that it is an isomorphism in the limit.
\end{proof}

\section{A stable multiplicative structure}
\medskip

We begin with the observation that the spaces $\underline{\E(n)}_0$ are $|v_n|=(2^n-1)(1+\alpha)$-periodic. That is, the adjoint of the multiplication-by-$v_n$ map on the spectrum $\E(n)$ induces, on the 0-space level, an equivalence
$$\xymatrix{
\sp{\E(n)}{0} \ar@{->}[r]^-{\simeq} & \Omega^{|v_n|}\sp{\E(n)}{0}
}$$
 We now express the spectrum $\E(n)$, and its products as colimits of shifted suspension spectra
\[ \E(n) = \colim_m \Sigma^{-m(2^n-1)(1+\alpha)} \underline{\E(n)}_0, \quad \E(n)^{\wedge k} = \colim_m \Sigma^{-mk(2^n-1)(1+\alpha)} \underline{\E(n)}_0^{\wedge k},  \]
where the maps are given by successive $v_n^{\wedge k}$ multiplications 
\[ \Sigma^{-mk|v_n|} \underline{\E(n)}_0^{\wedge k} \longrightarrow \Sigma^{-(m+1)k|v_n|}\underline{\E(n)}_0^{\wedge k}.\]
Applying $\E(n)$ cohomology, Milnor's $\lim^1$-sequence gives us
\[ 0 \longrightarrow \text{lim}^1 \E(n)^{-1}(\Sigma^{-mk|v_n|}\sp{\E(n)}{0}^{\wedge k}) \longrightarrow \E(n)^0(\E(n)^{\wedge k}) \longrightarrow \text{lim}^0 \E(n)^0(\sp{\E(n)}{0}^{\wedge k}) \longrightarrow 0.\]
Using the $|v_n|$-periodicity of $\E(n)$ noted above to identify $\Sigma^{-|v_n|}\underline{\E(n)}_0$ with $\underline{\E(n)}_0$, we have the sequence
\[ 0 \longrightarrow \text{lim}^{1} \E(n)^{-1} (\underline{\E(n)}_0^{\wedge k}) \longrightarrow \E(n)^0(\E(n)^{\wedge k}) \longrightarrow \text{lim}^0 \E(n)^0(\underline{\E(n)}_0^{\wedge k}) \longrightarrow 0. \]

\medskip
\noindent
We may now invoke the weak projective property of the spaces $\underline{\E(n)}_0^{\wedge k}$ (see Remark \ref{smashproj}) and identify the last term with $\lim^0 E(n)^0(\underline{E(n)}_0^{\wedge k})$. The space $\sp{E(n)}{0}$ has evenly graded cohomology (since it has the projective property) and homotopy; it follows from the Atiyah-Hirzebruch spectral sequence that the $E(n)$-cohomology of $\sp{E(n)}{0}$ is evenly graded as well (as are its $k$-fold smash products). Thus, in the analogous non-equivariant Milnor sequence 
\[ 0 \longrightarrow \text{lim}^{1} E(n)^{-1} (\underline{E(n)}_0^{\wedge k}) \longrightarrow E(n)^0(E(n)^{\wedge k}) \longrightarrow \text{lim}^0 E(n)^0(\underline{E(n)}_0^{\wedge k}) \longrightarrow 0 \] 
the $\lim^1$ term vanishes. Consequently, we identify $\lim^0E(n)^0(\underline{E(n)}_0^{\wedge k})$ with $E(n)^0(\sp{E(n)}{0})$ and obtain our short exact sequence of interest 
\begin{equation}\label{ses1} 0 \longrightarrow \lim{}^1 \E(n)^{-1} (\underline{\E(n)}_0^{\wedge k})  \longrightarrow \E(n)^0(\E(n)^{\wedge k}) \longrightarrow E(n)^0(E(n)^{\wedge k}) \longrightarrow 0,\end{equation}
with the last map being $\rho$. 
\medskip
\noindent

By writing $\E(n)^{\wedge k}$ as a colimit of $v_n^k$-multiplication maps as above and $\MU^{\wedge s}$ as a colimit of suspension spectra 
$$\MU^{\wedge s}=\colim_m\Sigma^{\infty-ms|v_n|}\underline{\MU}_{m|v_n|}^{\wedge s}$$
the above proof readily extends to show the existence of a short exact sequence
\begin{align}\label{ses2} 0 \longrightarrow \text{lim}^1 \E(n)^{-1} (\underline{\E(n)}_0^{\wedge k} \wedge \underline{\MU}_{m|v_n|}^{\wedge s})  \longrightarrow &\E(n)^0(\E(n)^{\wedge k} \wedge \MU^{\wedge s})& \\
 \quad \quad \quad \quad \quad \quad \quad \quad &\longrightarrow E(n)^0(E(n)^{\wedge k} \wedge MU^{\wedge s}) \longrightarrow 0. \nonumber \end{align}
where we have used Lemma \ref{zeroproj} above to identify the right hand term as before.
\medskip
\noindent

We are now ready to construct the $\MU$-algebra structure on $\E(n)$ that will be shown to be homotopy commutative and homotopy associative up to phantom maps. 

\medskip
\begin{defn} \label{ring structure}
Define $\hat{\mu}$ to be any element in $\E(n)^0(\E(n)^{\wedge 2})$ that lifts the canonical ring structure of $E(n)^0(E(n)^{\wedge 2})$ along $\rho$. Define the unit map $1 : \MU \longrightarrow \E(n)$ to be the canonical map expressing $\E(n)$ as a localized quotient of $\MU$. 
\end{defn}

\medskip
\noindent
As a formal consequence of the short exact sequence constructed above, we obtain Theorem \ref{main1} from the introduction:

\medskip
\begin{main1}
The class $\hat{\mu}$ defines a homotopy commutative, homotopy associative, and unital $\MU$-algebra structure on $\E(n)$ up to phantom maps.
\end{main1}
\begin{proof} By construction, $\hat{\mu}$ maps to the canonical ring structure on $E(n)$ under the forgetful map, $\rho$. It follows that any obstruction to the homotopy associativity, commutativity, or unitality of $\hat{\mu}$, viewed as a class in $\E(n)^0(\E(n)^{\wedge k} \wedge \MU^{\wedge s})$, maps to zero under $\rho$. We claim that the only elements of the kernel of $\rho$ are phantoms.

To see this, recall the short exact sequence (\ref{ses2}) above:
\begin{equation*} 0 \rightarrow \text{lim}^1 \E(n)^{-1} (\underline{\E(n)}_0^{\wedge k} \wedge \underline{\MU}_{m|v_n|}^{\wedge s})  \rightarrow \E(n)^0(\E(n)^{\wedge k} \wedge \MU^{\wedge s}) \rightarrow E(n)^0(E(n)^{\wedge k} \wedge MU^{\wedge s}) \rightarrow 0 \end{equation*}
where the right hand side was identified via
$$\text{lim}^0\E(n)^0(\underline{\E(n)}_0^{\wedge k} \wedge \underline{\MU}_{m|v_n|}^{\wedge s}) \cong \text{lim}^0E(n)^0(\underline{E(n)}_0^{\wedge k} \wedge \underline{MU}_{m|v_n|}^{\wedge s}) \cong E(n)^0(E(n)^{\wedge k} \wedge MU^{\wedge s}).$$
Notice that by exactness, an element in the kernel of $\rho$ is in the image of the $\lim^1$-term in $\E(n)^0(\E(n)^{\wedge k} \wedge \MU^{\wedge s})$ and so restricts trivially to all the terms $\E(n)^0(\underline{\E(n)}_0^{\wedge k} \wedge \underline{\MU}^{\wedge s}_{m|v_n|})$ in the inverse system. Any map from a finite CW complex into $\E(n)^{\wedge k} \wedge \MU^{\wedge s}$ must factor through a finite stage of the colimit, that is, through $\underline{\E(n)}_0^{\wedge k} \wedge \underline{\MU}^{\wedge s}_{m|v_n|}$ for some $m$. It follows that any element in the image of the $\lim^1$ term is zero upon restriction to any finite CW-complex. In other words, the kernel of $\rho$ consists entirely of phantoms.
\end{proof}

\medskip
\begin{remark}
One may attempt to compute the group of phantom maps $\lim{}^1 \E(n)^{-1} (\underline{\E(n)}_0^{\wedge k})$ explicitly by identifying it with the vector space $E(n)^{-1}(E(n)^{\wedge k}) \otimes \ZZ/2 $ (suitably extended by a $\ZZ/2$-agebra). This is an open problem, but it appears to the authors that this vector space is trivial for $n=1$, but may fail to be so for $n>1$. Hence we at present have no general way of ensuring that the ring structure we have constructed is rigid up to homotopy for the spectra $\E(n)$, $n>1$.
\end{remark}

\section{Multiplicative structure on $\MU_{(2)}[v_n^{-1}]$} \label{commutative MU}

\medskip
\noindent
We would like to show that the multiplication $\hat{\mu}$ constructed above naturally induces a commutative algebra structure on the $\E(n)$-cohomology of \emph{any} (not necessarily finite-dimensional) space. An essential ingredient in our construction will be the multiplication on $\MU_{(2)}[v_n^{-1}]$. We pause to describe it in this section. The ingredient we need is the following proposition, which appears as Proposition 9.15 in \cite{HHR16} and is proved in \cite{HH13}.

\medskip
\begin{prop}\label{prop:hhr}\cite[Corollary 4.11]{HH13} Let $R$ be a $G$-equivariant commutative ring with $D \in \pi_\star^G(R)$. If $D$ has the property that for every $H \subset G$, $N_H^Gi_H^*D$ divides a power of $D$, then the spectrum $D^{-1}R$ has a unique commutative algebra structure such that the map $R \longrightarrow D^{-1}R$ is a map of commutative rings.
\end{prop}

We begin by constructing a $C_2$-equivariant associative and commutative ring (in the highly structured sense) that lifts $\E(n)$. We begin with the fact that $\MU$ has this structure (see \cite{HHR16} or \cite{HK01}). We localize at $p=2$. The spectrum $\MU_{(2)}$ is a $C_2$-equivariant commutative ring, as shown in \cite{HH13}. By \cite{HK01}, the forgetful map
$$\rho: \pi_{\ast(1+\alpha)}(\MU_{(2)}) \longrightarrow \pi_{2\ast}(MU_{(2)})$$
is an isomorphism. The classes $v_i$ (Araki, Hazewinkel, or others) in $\pi_{2\ast}(MU_{(2)})$ may now be lifted via $\rho^{-1}$ to equivariant classes, $\rho^{-1}(v_i)$. While in the rest of the manuscript, we abuse notation by denoting $\rho^{-1}(v_i)$ by $v_i$, in the following lemma, we will distinguish between the nonequivariant $v_i$ and the equivariant $\rho^{-1}(v_i)$. Note that for $i\leq n$ and the Araki $v_i$, the images of these lifts $\rho^{-1}(v_i)$ in the coefficients of $\E(n)$ are exactly the equivariant $v_i$ we have been working with throughout.

Our next step is to invert $\rho^{-1}(v_n)$.

\medskip
\begin{lemma} The spectrum $(\rho^{-1}(v_n))^{-1}\MU_{(2)}$ is a $C_2$-equivariant commutative ring. \end{lemma}
\begin{proof} We apply Proposition \ref{prop:hhr}  above (quoted from \cite{HHR16}). The map $i_H^*$ is exactly $\rho$, and so $\rho(\rho^{-1}(v_n))=v_n \in \pi_{2(2^n-1)}(MU_{(2)})$. We need to show that $N_{\{e\}}^{C_2}(v_n)$ divides a power of $\rho^{-1}(v_n)$. In fact, we claim that $N_{\{e\}}^{C_2}(v_n)=-[\rho^{-1}(v_n)]^2$. To see this, we apply the isomorphism $\rho$ to both sides. Let $c$ denote the action of the generator of $C_2$. Using the fact that $c(v_n)=-v_n$, the double coset formula (see e.g. Proposition 10.9(v) in \cite{Sch15} or \cite{May96}) reduces in our case to
$$\rho \circ N_{\{e\}}^{C_2}(v_n)=v_n \cdot c(v_n)=-v_n^2,$$
which completes the proof.
\end{proof}

\medskip
\begin{remark}
Let us again denote the spectrum $(\rho^{-1}(v_n))^{-1}\MU_{(2)}$ by $\MU[v_n^{-1}]$. This spectrum serves as a commutative proxy for $\E(n)$. Indeed, essentially all prior results of the authors that hold for $\E(n)$ extend verbatim to this spectrum. 
\end{remark}

\section{Unstable properties of the multiplicative structure}
\medskip
\noindent
We now address the question of the (unstable) multiplicative structure on $\E(n)$. We begin with the following lemma:

\medskip
\begin{lemma}\label{zerospace}
There is a unique (homotopy) commutative, associative, and unital equivariant H-ring structure on the infinite loop space of $\E(n)$ that lifts any fixed H-ring structure on $\underline{E(n)}_0$:
$$\hat{\mu}_0 : \underline{\E(n)}_0 \times \underline{\E(n)}_0 \longrightarrow \underline{\E(n)}_0$$
\end{lemma}
\begin{proof} Recall that Lemma \ref{zeroproj} shows that $\sp{\E(n)}{0}^{\times j}$ has the weak projective property. As $E(n)$ is a (homotopy) associative and commutative ring spectrum, we may define the map $\hat{\mu}_0: \underline{\E(n)}_0 \times \underline{\E(n)}_0 \longrightarrow \underline{\E(n)}_0$ as the preimage of the multiplication on $\underline{E(n)}_0$ along the isomorphism
$$\xymatrix{\E(n)^0(\underline{\E(n)}_0\times \underline{\E(n)}_0) \ar@{->}[r]^-{\rho}_-{\cong} & E(n)^0(\underline{E(n)}_0 \times \underline{E(n)}_0)}$$
The unity, commutativity, and associativity of $\hat{\mu}_0$ are similarly verified by applying the isomorphism $\rho$, as the desired relations all hold in $E(n)$-cohomology. Likewise, given a choice of $H$-ring structure on $\sp{E(n)}{0}$, the uniqueness of the equivariant lift follows from the fact that $\rho$ is an isomorphism.
\end{proof}

\medskip
\begin{remark}
The unstable multiplication $\hat{\mu}_0$ we constructed in the previous lemma and the stable multiplication $\hat{\mu}$ we constructed in Definition \ref{ring structure} are compatible in the sense that $\hat{\mu}$ restricts to $\hat{\mu}_0$ on the zero space of $\E(n)$. To see this, apply the isomorphism $\rho$ and note that this claim is true nonequivariantly by construction.
\end{remark}

\medskip

\medskip
We now prove our second main result, Theorem \ref{main2}.

\begin{proof}(of Theorem \ref{main2}) Let $X$ be a space and consider $f \in \E(n)^V(X)$ and $g \in \E(n)^W(Y)$. We define the product $fg
\in \E(n)^{V+W}(X)$ as follows. First, note that since $\E(n)$ is an $\MU[v_n^{-1}]$-module, we may multiply $f$ and $g$ by classes in the coefficients $\MU[v_n^{-1}]_\star$. Let $k$ and $l$ be the minimal integers such that
$$v_n^kf \in \E(n)^{V'}(X), \ind v_n^lg \in \E(n)^{W'}(X)$$
with $V', W' \leq 0$ (by this we mean that when we express each representation as a combination of irreducibles, each coefficient should be nonpositive). These classes are represented by maps
$$\xymatrix{X \ar@{->}[r]^-{v_n^kf} & \underline{\E(n)}_{V'} = \Omega^{-V'}\underline{\E(n)}_0},
 \ind \xymatrix{X \ar@{->}[r]^-{v_n^lg} & \underline{\E(n)}_{W'} = \Omega^{-W'}\underline{\E(n)}_0}$$
We adjoin the loops over to form classes
$$\xymatrix{\Sigma^{-V'}X \ar@{->}[r]^-{v_n^kf} & \underline{\E(n)}_0}, \ind \xymatrix{\Sigma^{-W'}X \ar@{->}[r]^-{v_n^lg} & \underline{\E(n)}_0}$$
Note that since $-V', -W' \geq 0$, these are positive suspensions and so the sources of these maps are \emph{spaces}. We may thus smash them together, precompose with the diagonal on $X$ and postcompose with the multiplication on the zero space (which factors through the smash product) constructed in Lemma \ref{zerospace}:
$$\xymatrix{\Sigma^{-V'-W'}X \ar@{->}[r]^-{\Delta} & \Sigma^{-V'} X \wedge \Sigma^{-W'}X \, \, \ar@{->}[r]^-{v_n^kf \wedge v_n^lg} & \, \, \underline{\E(n)}_0 \wedge \underline{\E(n)}_0 \ar@{->}[r]^-{\hat{\mu}_0} & \underline{\E(n)}_0}$$
This produces a class in $\E(n)^{V'+W'}(X)$. Finally, we multiply by $v_n^{-k-l}$ to define the product $f\cdot g \in \E(n)^{V+W}(X)$.

\medskip
\noindent
The unit of this multiplication comes from the unit on $\MU[v_n^{-1}]$. The unity, associativity, and commutativity of this multiplication follow from the corresponding properties of $\underline{\E(n)}_0$.

\medskip
\noindent
We have shown that for any space $X$, $\E(n)^\star(X)$ is a graded associative and commutative ring. It remains to show that this is compatible with the multiplication on $\MU[v_n^{-1}]$-cohomology. If we carry out the above construction to define multiplication on $\MU[v_n^{-1}]^\star(X)$, it is evident that this agrees with the multiplication coming from the ring spectrum structure on $\MU[v_n^{-1}]$. To see that this multiplication agrees with the one defined on $\E(n)$-cohomology, it suffices to check this fact on zero spaces. To see that this diagram
$$\xymatrix{\sp{\MU}{0} \times \underline{\MU}_0 \ar@{->}[r]^-{\hat{\mu}_{\MU}} \ar@{->}[d] & \underline{\MU}_0 \ar@{->}[d] \\
\underline{\E(n)}_0 \times \underline{\E(n)}_0 \ar@{->}[r]^-{\hat{\mu}_{\E(n)}} & \underline{\E(n)}_0}$$
commutes, we may use the fact that $\underline{\MU}_0$ and $\underline{\MU}_0 \times \underline{\MU}_0$ are spaces with weak projective properties to map the diagram isomorphically along $\rho$ where its commutativity is apparent.
\end{proof}



\section{The MR$(n)$ orientation for MO$[2^{n+1}]$ revisited} 

\medskip
\noindent
Let $n > 0$ and let $\MO[2^{n+1}]$ denote the Thom spectrum for the virtual bundle over $\BOg$ given by multiplication by $2^{n+1}$ seen as a self-map of $\BOg$. In other words, $\MO[2^{n+1}]$ is the spectrum that represents real vector bundles $\xi$ endowed with an isomorphism $\xi \rightarrow 2^{n+1} \zeta$ for some bundle $\zeta$. 

\medskip
\noindent
In \cite{KW14}, the first and third authors showed that $\MO[2^{n+1}]$ admits an orientation with respect to $ER(n)$. However, that proof requires the homotopy commutativity of $\ER(n)$ which is unclear in light of this document. Consequently, in this section, we reproduce the argument in \cite{KW14} to show that $\MO[2^{n+1}]$ admits a canonical orientation with respect to the commutative ring spectrum $\MUR(n)$ defined as the homotopy fixed points of the spectrum $\MU [v_n^{-1}]$, where it is understood that $\MU$ is 2-local. This orientation descends to an orientation with respect to $ER(n)$, generalizing the $\hat{A}$-genus for real K-theory. We also take this opportunity to fix an error in the proof of this result given in \cite{KW14}, see Remark \ref{error}. 

\bigskip
\begin{thm} \label{main} 
The spectrum $\MO[2^{n+1}]$ supports a canonical orientation $u_{n+1}$ with respect to $\MUR(n)$. Furthermore, given a real vector bundle of the form $2^{n+1} \zeta$, this orientation is uniquely determined by the property that the image of $u_{n+1}$ in $\MUg[v_n^{-1}]$ is given by
\[ u_{n+1}(2^{n+1} \zeta) = \mu(2^n \zeta \otimes \CC) \cup \psi(\zeta \otimes \CC)^{2^{n-1}}, \]
where $\mu$ is the usual Thom class from usual complex orientation, and $\psi(\zeta \otimes \CC)$ is the series in $\MUg[v_n^{-1}]^*(BU)$ generated from line bundles by:
\[ \psi(x) = \frac{[-1]_{\MUg} (x)}{-x}. \]
\end{thm}

First, we lay some groundwork. 

Consider the Real orientation of $\MU$ 
given by a $\Z/2$-equivariant map: 
\begin{equation} 
\label{leftright}
\mu_1 : \MU(1) \longrightarrow \Sigma^{(1+\alpha)} \MU. 
\end{equation} 
where $\MU(1) \simeq \mathbb{CP}^\infty$ denotes the $C_2$-space in the usual prespectrum defining $\MU$. We view $MU^*(MU(1))$ as a rank one free module on $\mu_1$ over $MU^*(BU(1))$. We need the following fact regarding the $C_2$-action on $\mu_1$.

\bigskip
\begin{lemma}\label{action} Let $x$ denote the first Chern class in $MU^2(BU(1))$ and let $\psi(x)$ denote the series defined in Theorem \ref{main} above. Then
$$
c(\mu_1) = \mu_1 \frac{[-1]_{\MUg} (x)}{-x} 
= 
 \mu_1 \, \psi(x). 
$$

\end{lemma}
\begin{proof} The action of complex conjugation on $\Sigma^{1+\alpha} 
\MU$ can be identified with $-c$ (the $c$ from $\MU$
and the $-1$ from the orientation reversing action on the
two sphere). On the other hand, complex conjugation on $\MUg(1)$ is induced 
by the (complex anti-linear) self map of the universal line 
bundle $\gamma_1$ over $\BUg(1)$ that sends a vector to its 
complex conjugate. This map can be seen as a (complex linear) 
isomorphism from $\overline{\gamma_1}$ to $\gamma_1$, 
where $\overline{\gamma_1}$ is the opposite complex 
structure on the real bundle underlying $\gamma_1$. 
Since $\overline{\gamma}_1$ is isomorphic to the dual 
bundle $\gamma_1^*$, we see that the action of complex 
conjugation on $\MUg(1)$ sends the Thom class $\mu_1 \in 
\MUg^2(\MUg(1))$ to the class $[-1]_{\MUg} (\mu_1)$. 
$\MUg(1)$ and $\BUg(1)$ are homotopy equivalent and the Thom
isomorphism is an $\MUg^*(\BUg(1))$ module map so
that $\mu_1^2 = \mu_1 \, x$, 
where $x = c_1(\gamma_1)$. 

From this we have $\mu_1^k = \mu_1 x^{k-1}$, so any power series  $\sum a_i \mu_1^{i+1}$
can be rewriten as $\mu_1 \, \sum a_i x^i$.
Hence 
$$[-1]_{\MUg} (\mu_1) = \mu_1 \frac{[-1]_{\MUg} (x)}{x}.$$

Incorporating this observation into the $\Z/2$ 
equivariance of $\mu_1$ from Equation \eqref{leftright}
by computing on the left and the right, this
translates to the equality: 
$$
c(\mu_1) = \mu_1 \frac{[-1]_{\MUg} (x)}{-x} 
= 
 \mu_1 \, \psi(x). 
$$
\end{proof} 

\noindent

\textbf{Outline of proof:} (of Theorem \ref{main}) Before beginning the proof of Theorem \ref{main} in earnest, we give a brief overview. We will make significant use of the Bockstein spectral sequence constructed from the fibration in \cite[Theorem 1.6]{KW07a}. Though the results in \cite{KW07a} are stated for $ER(n)$, since the class $v_n$ has been inverted, they will apply verbatim to $MR(n)$. In particular, the proof of \cite[Theorem 1.6]{KW07a} shows that there is a Bockstein spectral sequence $E_r(\MO[4])$, starting with the $\MUg[v_n^{-1}]$-cohomology of the spectrum $\MO[4]$ and converging to the $MR(n)$-cohomology of $\MO[4]$. The proof of Theorem \ref{main} is somewhat technical and assumes familiarity with the properties of this spectral sequence, which may be found summarized in \cite[Theorem 2.1]{KW14}. In particular, note that the spectral sequence collapses at the $E_{2^{n+1}}$ page. The final part of the proof will make use of the structure of the spectral sequence for a point and for $\BOg$. These are described in Section 3 and 5 of \cite{KW14}, respectively.

The orientation $u_{n+1}$ will be constructed inductively, starting with a class $u_2 \in \MUg[v_n^{-1}]^*(\MO[4])$. Once the class $u_2$ is constructed, we will inductively define $u_{k+1}$ in terms of $u_k$ (see Equation \ref{ind} below) and begin our analysis of the Bockstein spectral sequence differentials. We will show that for $1<k<n+1$ that the class $u_k$ survives to the $E_{2^k-1}$ page and that $E_{2^k-1}(\MO[2^k])$ is a rank one free module over the $E_{2^k-1}(\BOg)$ (the $E_{2^k-1}$ page of the spectral sequence for $\BOg$) on the distinguished generator $u_k$. Continuing in this way, we will conclude that $u_{n+1}$ survives to the last stage $E_{2^{n+1}-1}$ where there is one last possible differential left in the spectral sequence. We will show that this differential on $u_{n+1}$ must be zero, and thus $MR(n)^*(\MO[2^{n+1}])$ is a rank one free module over $MR(n)^*(BO)$, which will establish the theorem.

\medskip
\noindent
\begin{proof}  (of Theorem \ref{main}) We start by defining a class $u_2$. Let $\zeta$ denote the universal real vector bundle over $\BOg$, $\zeta_{\CC}$ the universal complex vector bundle over $\BUg$, and $\overline{\zeta_{\CC}}$ its conjugate. Consider the composite
$$\xymatrix{\BUg \ar@{->}[r]^-{\Delta} & \BUg \times \BUg \ar@{->}[r] & \BUg}$$
classifying the bundle $\zeta_{\CC} \oplus \overline{\zeta_{\CC}}$ over $\BUg$. Precomposing with the complexification map $\BOg \longrightarrow \BUg$, we have that $\zeta_{\CC} \oplus \overline{\zeta_{\CC}}$ pulls back to $4\zeta$. Taking Thom spectra and mapping into $\MUg[v_n^{-1}]$ yields the composite
$$\MO[4] \longrightarrow \MUg \wedge \MUg \longrightarrow  \MUg[v_n^{-1}]$$
where the second map is induced by the twisted multiplication map (the counit of the norm-forgetful adjunction)
\[ m \circ (id \wedge c) : \MUg \wedge \MUg \longrightarrow \MUg. \]
This is an $\MUg[v_n^{-1}]$-Thom class for $4\zeta$, and we define it to be $u_2(4\zeta)$.
We have
$$u_2(4\zeta)=\mu((\zeta \otimes \CC) \oplus (\overline{\zeta \otimes \CC}))=\mu(\zeta \otimes \CC)\mu(\overline{\zeta \otimes \CC})$$
Reasoning from line bundles  using Lemma \ref{action} and applying the splitting principle shows that 
$$\mu(\overline{\zeta \otimes \CC})=\mu(\zeta \otimes \CC)\psi(\zeta \otimes \CC)$$
It follows that 
\[ u_{2}(4 \zeta) = \mu(2 \zeta \otimes \CC) \cup \psi(\zeta \otimes \CC)). \]
and that $u_2$ is $c$-invariant.


We now begin our investigation of the Bockstein spectral sequence. Since $u_2$ is $c$-invariant, it follows that $d_1(u_2)=0$ in the Bockstein spectral sequence converging to $\MUR(n)^*(\MO[4])$. We also have $d_2(u_2)=0$ for degree reasons, so that $d_r(u_2)=0$ for $r<3$ will begin our induction. If some Thom class $u_r$ (to be defined inductively in Equation \ref{ind} below) survives to $E_{r+1}$, notice that $\mbox{E}_{r+1}$ for the Thom space is a rank one free module over $\mbox{E}_{r+1}(\BOg)$ on the Thom class. We will show for $1<k<n+1$ that $E_{2^k-1}(\MO[2^k])$ is a rank one free module over $E_{2^k-1}(\BOg)$ on a distinguished generator $u_k$. 

By induction, assume this is true for $k>1$. Writing $2^{k+1}\zeta$ as $2(2^k\zeta)$, we have a factorization:
\[ [2^{k+1}] = ([2^k] + [2^k]) \circ \Delta : 
\BOg \longrightarrow \BOg \times \BOg \longrightarrow \BOg, \]
which induces a map of Thom spectra:
\[ \tau = m \circ \Delta : \MO[2^{k+1}] \longrightarrow \MO[2^k] 
\wedge \MO[2^k] \longrightarrow \MO[2^k]. \]
This induces a map of spectral sequences $\tau^\ast : 
\mbox{E}_{2^{k}-1}(\MO[2^k]) \longrightarrow 
\mbox{E}_{2^k-1}(\MO[2^{k+1}])$. 
Define 
\begin{equation}\label{ind}u_{k+1}:=\tau^\ast(u_k).\end{equation}
Observe that since $c(u_{k+1})=u_{k+1}$ by naturality and the fact that $c(u_k)=u_k$ by induction (the observation that $u_2$ is $c$-invariant above started the induction). By the properties of the Bockstein spectral sequence, we have: 
\[ d_{2^k-1}(u_{k+1}) = \Delta^\ast \mu^\ast 
d_{2^{k}-1}(u_k) = \Delta^\ast d_{2^{k}-1} \mu^\ast (u_k) = 
\Delta^\ast (d_{2^{k}-1}(u_k) \wedge u_k + u_k 
\wedge d_{2^{k}-1}(u_k)). \]
But notice that $\Delta : \MO[2^{k+1}] 
\longrightarrow \MO[2^k] \wedge \MO[2^k]$ is invariant 
under the swap map on $\MO[2^k] \wedge \MO[2^k]$. 
Therefore $\Delta^\ast (d_{2^{k}-1}(u_k) \wedge u_k) = 
\Delta^\ast (u_k \wedge d_{2^{k}-1}(u_k))$. 
It follows that $d_{2^{k}-1}(u_{k+1})$ is a 
multiple of $2$ and must consequently be zero 
since $\mbox{E}_{2^{k}-1}(\MO[2^{k+1}])$ is a $\Z/2$-module 
for external degrees greater than zero \footnote{Proposition \ref{bound} below shows that $d_{2^k-1}(u_k)$ is in fact non-trivial for $k \leq n$}. It follows that $u_{k+1}$ survives 
to $\mbox{E}_{2^{k}}(\MO[2^{k+1}])$. 

\medskip
\noindent
We now need to show that $d_{2^{k}+r}(u_{k+1})=0$ by induction
on $r$, for $0 \le r < 2^{k}-1$. Here we use a degree argument that relies on the structure of the Bockstein spectral sequence as described in Section 5 of \cite{KW14} and we import our notation from there, in particular the `hatted' classes below. We write the Bockstein spectral sequence for $\BOg$ as a tensor product of a ring of permanent cycles and the Bockstein spectral sequence for a point, as in \cite{KW14}:
$$E_r(\BOg)=\hat{\MUg}[\hat{v}_n^{-1}]^*(\BOg) \otimes_{\hat{\MUg}[\hat{v}_n^{-1}]^*} E_r(\text{pt})$$
By construction, all classes in $\hat{\MUg}[v_n^{-1}]^\ast(\BOg)$ have internal 
degree divisible by $2^{n+2}$. On the other hand, the differentials 
longer than $d_{2^k-1}$ have external degree 
larger than $2^k-1$, and hence represent elements divisible by $x^{2^k-1}$. 

 Using the structure of the spectral sequence for a point, the domain of these differentials is generated by the 
classes $ y, \hat{v}_{i,l}:=\hat{v}_iv_n^{l2^{i+1}}$ and $v_n^{\pm 2^k}$ for $i > k-1$, where $y$ is the permanent cycle representing the nilpotent class $x$. All of these classes have internal degree divisible by $2^{k+1}$. 
Therefore, for dimensional reasons, there can be no 
differentials in this spectral sequence that land in internal 
degree between $2^k$ and $2^{k+1}$, until we reach the 
differential $d_{2^{k+1}-1}$. 

\medskip
\noindent
Continuing in this way, we notice that $u_{n+1}$ survives until the last stage $\mbox{E}_{2^{n+1}-1}(\MO[2^{n+1}])$. Now consider $d_{2^{n+1}-1}(u_{n+1})$ in degree $1+(2^{n+1}-1) \lambda$. The image of this differential lands inside a subquotient of the group $\MUg[v_n^{-1}]^*(\MO[2^{n+1}])$, which we henceforth identify (using the Thom isomorphism) with the group $u_{n+1} \MUg[v_n^{-1}]^*(\BOg)$. Furthermore, classes that have survived past $\mbox{E}_{2^n-1}$ must belong to the $\Z/2[\hat{v}_1, \ldots, \hat{v}_{n-1}, v_n^{\pm 2^n}]$-submodule of $u_{n+1} \MUg[v_n^{-1}]^*(\BOg)$ generated by $u_{n+1} \hat{\MUg}[v_n^{-1}](\BOg)$, modulo previous differentials. This allows us to express $d_{2^{n+1}-1}(u_{n+1})$ as: 
\[ d_{2^{n+1}-1}(u_{n+1}) = v_n^{2^n} u_{n+1} \, w, \quad \quad  \text{where} \quad \quad w \in \frac{\MUg[v_n^{-1}]^*(\BOg)}{\langle \hat{v}_0, \ldots, \hat{v}_{n-1} \rangle},  \]
for some permanent cycle $w$. Furthermore, we know that $d_{2^{n+1}-1}^2 (u_{n+1}) = 0$. Applying this to the above expression and using the derivation property, we see that: 
\[ v_n^{2^{n+1}} w^2 = v_n^{2^{n+1}-2^{2n}} \hat{v}_n w. \]
Replacing $w$ with $v_n^{-2^{2n}} \hat{w}$ for some (unique) element $\hat{w} \in \MUg[v_n^{-1}]^{(1-2^n)(1-\lambda)}(\BOg)$, we obtain the relation $\hat{w}^2 = \hat{v}_n \hat{w}$. This implies that $\hat{w} (1- \hat{v}_n^{-1} \, \hat{w}) = 0$. Since $(1-\hat{v}_n^{-1}\hat{w})$ is a unit, we see that $\hat{w} = 0$. In other words, $u_{n+1}$ survives the differential $d_{2^{n+1}-1}$. The proof of Theorem \ref{main} is complete on observing that the 
spectral sequence collapses at $\mbox{E}_{2^{n+1}}$ (by \cite[Theorem 2.1(iv)]{KW14}.
\end{proof}

\medskip
\noindent
The following proposition shows that we cannot expect to do better: 

\medskip
\begin{prop} \label{bound}
There exists a vector bundle $\zeta$ such that $2^n \zeta$ is not $ER(n)$-orientable. In particular, the differential $d_{2^k-1}$ in the Bockstein spectral sequence converging to $\MUR(n)^*(\MO[2^k])$ is nontrivial on the generator $u_k$ for $k \leq n$. 
\end{prop}
\begin{proof}
Note that if $d_{2^k-1}$ was trivial on $u_k$ for some $k \leq n$, the proof of the above theorem would show that $\MO[2^m]$ was $\MUR(n)$-orientable for some $m \leq n$. Let us demonstrate a contradiction under that hypothesis by showing the existence of a vector bundle $\zeta$ such that $2^n \zeta$ is not $ER(n)$-orientable. Let $S^{\alpha}$ denote the one point compactification of the sign representation of $\Z/2$. Consider the virtual vector bundle $\zeta$ over $B\Z/2$ with Thom space given by: 
\[ Th(\zeta) = E\Z/2_+ \wedge_{\Z/2} S^{(\alpha-1)}. \] 
If $2^n\zeta$ were to admit an $\MUR(n)$-orientation, then one would obtain a map representing the Thom class $\mu$:
\[ \mu: E\Z/2_+ \wedge_{\Z/2} S^{2^n(\alpha-1)} \longrightarrow MR(n).\]
Postcomposing with the inclusion of fixed points map $MR(n) \longrightarrow \MU[v_n^{-1}]$ (where we view $MR(n)$ as a $C_2$-spectrum with trivial action) and precomposing with the map to the orbits, we have the composite
$$E\Z/2_+ \wedge S^{2^n(\alpha -1)} \longrightarrow E\Z/2_+ \wedge_{\Z/2} S^{2^n(\alpha-1)} \longrightarrow MR(n) \longrightarrow \MU[v_n^{-1}].$$
which is a $C_2$-equivariant map (with trivial actions on the middle two terms). Taking adjoints, we obtain a map
$$S^{2^n(\alpha -1)} \longrightarrow F(E\Z/2_+, \MU[v_n^{-1}])\simeq \MU[v_n^{-1}],$$
where we have used the fact that $\MU[v_n^{-1}]$ is cofree. Since it comes from the Thom class, this map must represent a unit in $\pi_{2^n(\alpha -1)}^{\Z/2}\MU[v_n^{-1}]$. However, the computation of the bigraded homotopy of $\MU[v_n^{-1}]$ given in \cite{HK01} shows that there is no such class. Hence we obtain a contradiction to the existence of $\mu$. 
\end{proof}

\medskip
\begin{remark}
As mentioned earlier, Theorem \ref{main} descends to an $ER(n)$-orientation for $\MO[2^{n+1}]$. However, this is by no means optimal. 
For example, for $n=1$ we know that $ER(1)$ is 2-localized real 
$K$-theory. Hence a bundle $\xi$ is $ER(1)$-orientable if and 
only if it is {\em Spin}. This is equivalent to $w_1(\xi) = w_2(\xi) = 0$. 
This holds for bundles of the form $\xi = 4\zeta$, 
but clearly there are Spin bundles that are not divisible by $4$. Similarly, for $ER(2)$, the results of \cite{KS} suggest that a bundle $\xi$ is $ER(2)$-orientable if and only if $w_1(\xi) = w_2(\xi) = w_4(\xi) = 0$, which is clearly true for bundles of the form $\xi = 8\zeta$. It is a compelling question to find a nice answer in 
general for when a bundle is $ER(n)$-orientable, or even to 
show that an answer to this question may be given in closed form. 
\end{remark}

\medskip
\begin{remark} \label{error}
The above Theorem \ref{main} corrects an error given in the proof of \cite[Theorem 1.4]{KW14}.  The induction process for the construction of $u_{n+1}$ in \cite{KW14} began with a class in the Bockstein spectral sequence converging to $ER(n)^\ast(\MO[2])$. Unfortunately, that class turns out to not be conjugation invariant as required. Our current argument starts with a manifestly invariant class in the Bockstein spectral sequence converging to $ER(n)^*(\MO[4])$ and use that to generate the other permanent cycles. The rest of the argument is essentially the same as in \cite{KW14}. 
\end{remark}

\bibliography{ernbib}{}
\bibliographystyle{alpha}
\end{document}